\documentclass[11pt,letterpaper]{amsart}
\usepackage{graphicx}
\usepackage{verbatim}
\usepackage{hyperref}
\newtheorem{theorem}{Theorem}[section]
\newtheorem{lemma}[theorem]{Lemma}

\theoremstyle{definition}
\newtheorem{definition}[theorem]{Definition}
\newtheorem{example}[theorem]{Example}

\newtheorem{prob}[theorem]{Problem}

\theoremstyle{remark}

\numberwithin{equation}{section}

\begin{document}

\title[Germs of extremal K\"ahler metrics]{The space of germs of extremal K\" ahler metrics in one dimension comprises three distinct ${\Bbb R}^3$ components}

\author{Qing Chen}
\address{School of Mathematical  Sciences, University of Science and Technology of China, Hefei 230026 China}
\email{qchen@ustc.edu.cn}

\author{Yiqian Shi}
\address{School of Mathematical Sciences, University of Science and Technology of China, Hefei 230026 China}
\email{yqshi@ustc.edu.cn}
%\thanks{}

%\author{Tianyang Sun}
%\address{School of Mathematical  Sciences, University of Science and Technology of China, Hefei 230026 China
%\curraddr{Department of Mathematics, UC Berkeley, Berkeley,  CA 94720 USA}
%\email{tysun@mail.ustc.edu.cn}

%\author{Jijian Song}
%\address{Center for Applied Mathematics, School of Mathematics, Tianjin University\newline \indent Tianjin, 300350, China}
%\email{smath@tju.edu.cn}
%\thanks{}

\author{Bin Xu$^\dagger$}
\address{School of Mathematical Sciences, University of Science and Technology of China, Hefei 230026 China}
\email{bxu@ustc.edu.cn}

\thanks{Q.C is supported in part by NSFC (Grant No. 11971450).
Y.S. is supported in part by NSFC
(Grant No. 11931009) and Anhui Initiative in Quantum Information Technologies (Grant No. AHY150200).
B.X. is supported in part by the Project of Stable Support for Youth Team in Basic Research Field, CAS (Grant No. YSBR-001) and NSFC (Grant Nos. 12271495, 11971450, and 12071449). }
\thanks{$^\dagger$B.X. is the corresponding author.}

%    General info

\date{}

\begin{abstract} 
In the 1980s, Eugenio Calabi introduced  the concept of {\it extremal K\" ahler metrics} as critical points of 
the $L^2$-norm functional of scalar curvature in the space of K\" ahler metrics belonging to a fixed K\"ahler class of a compact complex manifold $X$. Calabi demonstrated that extremal K\" ahler metrics always degenerate into Einstein metrics on compact Riemann surfaces. 
We define a K\"ahler metric $g$ on a domain of ${\Bbb C}^n$ as a {\it local extremal K\"ahler metric} of dimension $n$ if it satisfies the Euler-Lagrange equation of this functional, i.e. holomorphic is
the $(1,0)$-part of the gradient vector field of the scalar curvature of $g$,  in the domain.
Our main result establishes that the space of all germs of local extremal, non-Einstein K\"ahler
metrics of dimension one comprises three components, each diffeomorphic to ${\Bbb R}^3$.
 
\end{abstract}

\dedicatory{}

%\keywords{${\rm SU}(n+1)$ Toda system,  regular singularity, multi-valued holomorphic curve}

\maketitle

%%%%%%%%%%%%%%%%%%%%%%%%%%%%%%%%%%%%%%%%%%%%%%%%%%%%%%%%%%%%%%%%%%%%%%%%
\section{Introduction}
In Differential Geometry, a key challenge is finding canonical metrics on a given manifold, with various manifestations. Notably, the uniformization theorem for Riemann surfaces asserts that each Riemann surface admits a unique complete Kähler Einstein metric.  To find canonical K\" ahler metrics on a compact complex manifold $X$ of dimension $n$, in the 1980s,  Eugenio Calabi \cite{Ca1982, Ca1985} introduced the seminal concept of {\it extremal K\" ahler metric} as a critical point of the $L^2$-norm energy functional $\int_X S_g^2\, \omega_g^n$
of scalar curvature $S_g$ in the space of K\" ahler metrics $g$ belonging to a fixed K\" ahler class $[\omega_g]$ on $X$. 
Calabi's work initiated a rapidly growing literature, explored further in Gábor Székelyhidi's 2014 textbook \cite{Sze2014}  on extremal K\"ahler metrics.

Calabi demonstrated that  a K\"ahler metric $g=(g_{j\bar k})$ on $X$ is {\it extremal} if and only if the $(1,0)$-part 
${\rm grad}^{1,0} S_g:=\sum_{j,k}\,g^{j\bar k}\frac{\partial S_g}{\partial \overline{z_k}}\frac{\partial}{\partial z_j}$
of the Riemannian gradient vector field of the scalar curvature $S_g$ of $g$ is a holomorphic vector field on $X$.
(\cite[Theorem 2.1]{Ca1982} and \cite[Theorem 4.2]{Sze2014}). This holomorphic vector field is termed  the {\it complex gradient vector field} of $S_g$.
As a consequence, extremal K\" ahler metrics inherently possess constant scalar curvature provided that $X$ admits no nontrivial holomorphic vector field.
Furthermore, they become Einstein metric if K\"ahler class is proportional to $c_1(X)$. 
%On the other hand,  Calabi \cite[Theorem 3.1]{Ca1982} showed an interesting fact that extremal metrics  on compact Riemann surfaces always have %constant curvature, i.e. they are K\"ahler Einstein, which motivated us to write up this manuscript.
A K\"ahler metric $g=(g_{j\bar k})$ on a domain $\Omega\subset {\Bbb C}^n$ is defined as a {\it local extremal K\" ahler metric} of dimension $n$ if and only if ${\rm grad}^{1,0} S_g$
is holomorphic in $\Omega$. Fixing a point in $\Omega$, say $z=0$, we consider two local extremal K\"ahler metrics $g_1$ and $g_2$ in domains $\Omega_1\ni 0$ and $\Omega_2\ni 0$ respectively, to be {\it equivalent} if and only if $g_1$ is isometric to $g_2$ in some neighborhood of $0$. This is expressed
through  a holomorphic coordinate change $z\mapsto \Phi(z)$ near $z=0\in \Omega_1$, satisfying 
$\Phi(0)=0$ and $\Phi^* g_2=g_1$ near $z=0$. 
An equivalence class of local extremal K\"ahler metrics living in domains containing $0\in {\Bbb C}^n$ is termed  a {\it germ of local extremal K\"ahler metrics of dimension $n$}, or simply an {\it extK germ}. Analogously, similar concepts arise for 
 {\it germs of local K\"ahler Einstein metrics} and {\it germs of local K\"ahler metrics of constant scalar curvature} of dimension $n$.

We shall focus on (local) K\"ahler metrics of dimension one in the following. 
Joseph Liouville (\cite{Liou1853} and \cite[pp. 26-29]{Ban1980}) established that  all germs of local K\"ahler Einstein metric of dimension one  form the real one-parameter family of $\frac{4|{\rm d}z|^2}{\big(1+K_0|z|^2\big)^2},\,  K_0\in {\Bbb R}$. 
Additionally,  Eugenio Calabi \cite[Theorem 3.1]{Ca1982} demonstrated that extremal metrics on compact Riemann surfaces always have constant curvature.
Local K\"ahler extremal metrics of dimension one are closely connected to the research program of {\it extremal Hermitian metrics and HCMU metrics} on compact Riemann surfaces, initiated by Xiuxiong Chen \cite{Chen1999, Chen2000} in the 1990s to extend the uniformization theorem to Riemann surfaces with boundary. 
An {\it extremal Hermitian metric}  on a compact Riemann surface 
is a conformal metric $g$ with finitely many cone or cusp singularities prescribed a priori,  satisfying the Euler-Lagrange equation of the $L^2$-norm energy functional of Gaussian curvature in a suitable $H^{2,2}$ space of conformal metrics with prescribed singularities (\cite[(3)]{Chen1999} and \cite[(2.1-2)]{Chen2000}). 
This equation reads on the Riemann surface  
%$\begin{equation*}
%\label{eqn:extH_intr}
$\frac{\partial}{\partial \bar z}\, K_{g,\, zz}=0$, 
%\end{equation*}
i.e. the $(2,0)$-part of the real Hessian of the Gaussian curvature $K_g$ of $g$ is holomorphic, outside the singularities. 
An {\it HCMU metric} $g$ is 
an extremal Hermitian metric on a compact Riemann surface such that $K_{g,\,zz}$ vanishes outside the singularities of $g$
(\cite[(2.3)]{Chen2000}). 
In particular, restrictions of HCMU metrics near their smooth points provide us an abundance of examples of extK germs of dimension one.
%Wei-Wu \cite[Theorem 1.1.]{WW2018} demonstrated that such extK germs actually arise from HCMU footballs  (\cite[Section 8]{Chen2000}, \cite[p. 208]{WZ2000} and \cite[Section 2.2]{CCW2005}). 
Following the foundational results by Xiuxiong Chen \cite{Chen1999, Chen2000}, various authors, including  Wang-Zhu \cite{WZ2000}, Lin-Zhu \cite{LZ2002}, Chen-Chen-Wu \cite{CCW2005},  Chen-Wu \cite{CW2009, CW2011} and Chen-Wu-Xu \cite{CWX2015}, obtained diverse classification and existence results about extremal Hermitian and HCMU metrics on compact Riemann surfaces. In particular, HCMU and non-Einstein metrics were reduced to an integrable system by Lin-Zhu \cite[Lemma 2.1.]{LZ2002} and were classified by Chen-Wu \cite[Theorem 2.2.]{CW2011} and Chen-Wu-Xu \cite[Theorems 1.1-1.7.]{CWX2015}. 
%Zooming up HCMU metrics around their smooth points, we could obtain lots of germs of local K\"ahler extremal and non-Einstein metrics of dimension one, all of which
%actually arise from smooth points of HCMU footballs  (\cite[Section 2.2]{CCW2005}, \cite[Section 8]{Chen2000} and \cite[p. 208]{WZ2000}) by Wei-Wu \cite[Theorem 1.1.]{WW2018}. 
Motivated by the above,
we classify 1-dimensional extK germs, resulting in the following theorem:

\begin{theorem}
 The space ${\frak M}$ of all germs of 1-dimensional local K\"ahler extremal, non-Einstein metrics comprises three components, each diffeomorphic to ${\Bbb R}^3$.
\end{theorem}

We conclude this section by explaining how the left part of this manuscript is organized. 
In the next section,
depending on whether the Gaussian curvature of a germ achieves its extremal value at $z=0$ or not, we divide all germs in Theorem 1.1. into the two categories of
{\it exceptional} germs and {\it generic} ones. We prove in Section 3 two classification theorems (Theorems 3.1. and 3.2.) for these two classes, respectively, from which Theorem 1.1. follows immediately.  In the last section, we have a discussion about relevant open questions.

\section{Generic and exceptional germs} 
In this section, inspired by the work of Lin and Zhu \cite[Proposition 1.2]{LZ2002} and employing holomorphic coordinate transformations, we present Lemmas 2.1-2 that recast the partial differential equation $K_{g,,zz}=0$ satisfied by an extremal Kähler (extK) germ $g$ with nonconstant curvature. Lemma 2.2. yields two distinct differential systems, each comprising an ordinary differential equation involving $K_g$ and another equation expressing $g$ in terms of $K_g$. The choice between these two systems depends on whether the Gaussian curvature $K_g$ attains its extremal value at $z=0$ or not. This distinction facilitates the classification of germs into two categories: the {\it exceptional} and {\it generic} ones (as per Definition 2.3).

\begin{lemma}
\label{lem:hcmu_sm} 
Let $g=e^{2\phi}|{\rm d}w|^2$ be an extK germ at $w=0 \in {\Bbb C}$ and $K=K_g$ the Gaussian curvature of $g$ which is non-constant. Then 
\[0\not\equiv F(w)\frac{\partial }{\partial w}:=4K_{,w}\frac{\partial }{\partial w}=4e^{-2\phi}K_{\overline w}\frac{\partial }{\partial w}=
4e^{-2\phi}\frac{\partial K}{\partial {\overline w}}\frac{\partial }{\partial w}\]
is a holomorphic vector field near $w=0$. There exist two real numbers $C$ and $C'$ determined by $g$ such that  
\begin{equation}
\label{eqn:hcmus}
\left\{
\begin{split}
%K&=-\frac{4\phi_{z\bar z}}{e^{2\phi}}\\
\frac{1}{F}&=\frac{K_w}{-\frac{K^3}{3}+CK+C'}\\
e^{2\phi}&=\frac{4\left(-\frac{K^3}{3}+CK+C'\right)}{|F|^2}
\end{split}
\right.\,.
\end{equation}
Moreover, if $F(0)=0$, then{\rm :}
\begin{itemize}
%and the Gaussian curvature $K$ of $g$ achieves its local minimum (maximum) at $z=0$ provided $F'(0)>0\, (<0)${\rm ;} 
\item[(1)] $K_0:=K(0)$ is a  root of the cubic polynomial $-t^3/3+Ct+C'$; and 

\item[(2)] $F'(0)$ is a nonzero real number.
\end{itemize}
\end{lemma}

\begin{proof} The vector field $F(w)\frac{\partial }{\partial w}=4K_{,w}\frac{\partial }{\partial w}$ is a multiple of ${\rm grad}^{(1,0)} S_g$, which is holomorphic near $w=0$ 
by the definition of local K\"ahler extremal metric, i.e. 
\begin{equation}
\label{eqn:hcmu_C}
K_{,ww}=0\ \ {\rm near}\ \ w=0, \end{equation} 
from which we obtain the more general differential equation
\begin{equation}
\label{eqn:extH_C1}
\frac{\partial}{\partial \overline w}K_{,ww}=0\quad {\rm near }\quad w=0.
\end{equation}
$F(w)\frac{\partial }{\partial w}=4e^{-2\phi}K_{\overline w}\frac{\partial }{\partial w}$ does not vanish
identically since $K$ is non-constant. By \cite[Proposition 1, pp. 273-274]{Chen2000},  \eqref{eqn:extH_C1} is equivalent to the equation
\begin{equation}
\label{eqn:extH_C2}
4K_{w\overline w}=\big(C-K^2)e^{2\phi}\ \  {\rm near }\ \  w=0,
\end{equation}
where  $K_{w\overline{w}}:=\frac{\partial ^2 K}{\partial w\partial \overline w}$ and $C$ is some real constant determined by the germ $g$.

The validity of \eqref{eqn:hcmus} follows exactly from \cite[Lemma 2.1., pp. 188-190]{LZ2002} and \cite[Proposition 3.1, pp. 395-396]{CWX2015}. 
The details are as follows. Since $e^{2\phi}=\frac{4K_{\overline w}}{F}$ and $F$ is holomorphic, 
rewriting \eqref{eqn:extH_C2} as $K_{w\overline w}=\left(\frac{-K^3/3+CK}{F}\right)_{\overline w}$, we find that
$G:=K_w-\frac{-K^3/3+CK}{F}$ is holomorphic near $w=0$, and the holomorphic function $FG$ actually turns out to be some real constant, denoted by $C'$, since
$FG$ equals the real valued function $4e^{-2\phi}|K_w|^2+K^3/3-CK$ near $w=0$.
Hence, we obtain near $w=0$ 
\begin{equation} 
\label{eqn:hcmus_1}
K_w=\frac{-K^3/3+CK+C'}{F},\ \ {\rm i.e.}\ \ \frac{K_w}{-K^3/3+CK+C'}=\frac{1}{F}, 
\end{equation} 
and
\begin{equation}
\label{eqn:hcmus_2}  
e^{2\phi}=\frac{4K_{\overline w}}{F}=\overline{\left(\frac{-K^3/3+CK+C'}{F}\right)}\cdot\frac{4}{F} =\frac{4\left(-K^3/3+CK+C'\right)}{|F|^2},\end{equation}
both of which are identical to the system \eqref{eqn:hcmus}.

Assume that $F(0)=0$. Since $e^{2\phi}$ is smooth at $w=0$, $-K^3/3+CK+C'$ must vanish at $w=0$
by the preceding equation. 
%We shall show at the end of the proof that {\it $K_0:=K(0)$ is a simple root of $-t^3/3+Ct+C'$}.   
%According to the multiplicity of $K_0=K(0)$ as a root of the cubic polynomial
%$-t^3/3+Ct+C'$, we shall prove by contradiction in the following three scenarios that {\it $w=0$ is a simple zero of $F(z)$}.  
Considering the multiplicity of $K_0=K(0)$ as a root of the cubic polynomial $-t^3/3+Ct+C'$, we will establish, through proof by contradiction in three distinct scenarios, that {\it $w=0$ is indeed a simple zero of the function $F(z)$}.
Suppose that $w=0$ has multiplicity $n>1$ as a zero of $F$. 
By \cite[Theorem 6.4]{Str1984}, using holomorphic coordinate change, we could simplify the one-form 
$\frac{{\rm d}w}{F}$ into the form of $\left(\frac{1-n}{w^n}+\frac{b}{w}\right){\rm d}w$ for some $b\in {\Bbb C}$ 
near $w=0$. Since extK germs and their Gaussian curvature functions are intrinsic objects in Differential Geometry, we may assume that  $\frac {1}{F(w)}=
\frac{1-n}{w^n}+\frac{b}{w}$ in the system \eqref{eqn:hcmus} that $g$ and $K=K_g$ satisfy. By \eqref{eqn:hcmus_1}, $b$ must be real since
\[\frac{{\rm d}K}{-K^3/3+CK+C'}=\left(\frac{1-n}{w^n}+\frac{b}{w}\right){\rm d}w
+\overline{\left(\frac{1-n}{w^n}+\frac{b}{w}\right){\rm d}w}\]
is a real differential one-form. 
%$F(w)\frac{\partial}{\partial w}=w^n\frac{\partial}{\partial w}$. 
\begin{itemize}
\item[(i)] Assume that $K_0$ is a triple root of $-t^3/3+Ct+C'$. Then $K_0=C=C'=0$. Rewriting \eqref{eqn:hcmus_1} as 
$
-3\frac{{\rm d}K}{K^3}=\left(\frac{1-n}{w^n}+\frac{b}{w}\right){\rm d}w
+\overline{\left(\frac{1-n}{w^n}+\frac{b}{w}\right){\rm d}w}
$
and integrating it, we obtain near $w=0$
\begin{equation*}
 \frac{3}{2K^2}=w^{1-n}+{\overline w}^{1-n}+b\ln\, |w|^2+{\rm Const}.
\end{equation*}
Taking the limit inferior at $w=0$ on both sides of the above equation, we obtain
$+\infty=-\infty$ since $n>1$. Contradiction!

\item[(ii)] Assume that $K_0$ is a double root of $-t^3/3+Ct+C'$. Since
$\frac{1}{K(w,\overline w)-K_0}$  equals a nonzero multiple of 
$\big(w^{1-n}+{\overline w}^{1-n}+b\ln\, |w|^2\big)$
plus a minor term near $w=0$, by \eqref{eqn:hcmus_2}, we obtain 
$\limsup_{w\to 0}\, e^{2\phi}=+\infty$. Contradiction!

\item[(iii)] Assume that $K_0$ is a simple root of $-t^3/3+Ct+C'$. Since $\ln\, |K(w,\overline w)-K_0|$ equals a nonzero  multiple of 
$\left(w^{1-n}+{\overline w}^{1-n}+b\ln\,|w|^2\right)$ plus a bounded term near $w=0$,  we obtain the same contradiction as (2). 
\end{itemize}
Therefore we already proved that $\sigma^{-1}:=F'(0)\not=0$.
The fact that {\it $\sigma$ is real} follows from a similar argument as that for $b\in {\Bbb R}$ above.
%The fact that {\it $\sigma$ is real} follows from a similar argument as that for $b\in {\Bbb R}$. 

%At last, we show that {\it $K_0$ is a simple root of $-t^3/3+Ct=C'$, and
%$K$ achieves its local minimum (maximum) $K_0$ at $w=0$ if and only $\sigma>0$ ($\sigma<0$)}.
%Assuming that $F(0)=0$, we find that Statement (1) follows from the proof of \cite[Proposition 2.2., pp. 273-274]{CW2009} and (2) from \cite[Lemma %3.1. (c), pp. 333-334]{CW2011}.
\end{proof}

We could use holomorphic coordinate change to simplify \eqref{eqn:hcmus} that extK germs satisfy in the following:

\begin{lemma}
\label{lem:hcmu_sm_2} 
Let $g=e^{2\phi}|{\rm d}w|^2$ be an extK germ at $w=0\in {\Bbb C}$ and have
{\rm nonconstant} curvature. Using some suitable holomorphic coordinate change $w\mapsto z(w)$ near $w=0$ such that $z(0)=0$,
we have that $g$ and its Gaussian curvature $K$ satisfy
either 
\begin{equation}
\label{eqn:hcmus_o1}
\left\{
\begin{split}
%K&=-\frac{4\phi_{z\bar z}}{e^{2\phi}}\\
\partial K&=\left(-\frac{K^3}{3}+CK+C'\right){\rm d}z\\
g&=4\left(-\frac{K^3}{3}+CK+C'\right)|{\rm d}z|^2
\end{split}
\right.
\end{equation}
or 
\begin{equation}
\label{eqn:hcmus_o2}
\left\{
\begin{split}
%K&=-\frac{4\phi_{z\bar z}}{e^{2\phi}}\\
\partial K&=\sigma\left(-\frac{K^3}{3}+CK+C'\right)\frac{{\rm d}z}{z}\\
g&=4\sigma^2\left(-\frac{K^3}{3}+CK+C'\right)\frac{|{\rm d}z|^2}{|z|^2}\\
0&=-K_0^3/3+CK_0+C'\quad{\rm where}\quad K_0=K(0)
\end{split}
\right.\, 
\end{equation}
\end{lemma}
near $z=0$.
\begin{proof} Using notions in  Lemma \ref{lem:hcmu_sm} and its proof, we can see that
$\omega:=\frac{{\rm d}w}{F(w)}$ is a meromorphic one-form near $w=0$ and has at most simple pole at $w=0$ with real residue.
We rewrite the system \eqref{eqn:hcmus} as the following intrinsic form:
\begin{equation}
\label{eqn:hcmus_o}
\left\{
\begin{split}
%K&=-\frac{4\phi_{z\bar z}}{e^{2\phi}}\\
\partial K&=\left(-\frac{K^3}{3}+CK+C'\right)\omega\\
g&=4\left(-\frac{K^3}{3}+CK+C'\right)|\omega|^2
\end{split}
\right.
\end{equation}
Under some holomorphic coordinate change $w\to z(w)$ that preserves the origin,  
$\omega=\frac{{\rm d}w}{F(w)}$ could be simplified be the form of ${\rm d}z$ or $\sigma \frac{{\rm d}z}{z}$ with $\sigma\in {\Bbb R}^\times$.
Substituting this into \ref{eqn:hcmus_o}, we are done.
\end{proof}

\begin{definition}
We refer to an extK germ that satisfies \eqref{eqn:hcmus_o1} as {\it generic}, and one satisfying \eqref{eqn:hcmus_o2} as {\it exceptional}.
It is worth noting that neither the generic germ nor the exceptional one possesses  constant curvature.
\end{definition}

Given the pivotal role played by the cubic polynomial $-t^3/3 + Ct + C'$ in the subsequent discussion, we provide the Cardano formula \cite{Card} here for the convenience of the readers.

\begin{lemma}
\label{lem:Card}
Denote by $D:=27\big(4C^3-9(C')^2\big)$ the discriminant  of the cubic polynomial $-t^3/3+Ct+C'$ with 
$C,C'\in {\Bbb R}$. There hold the following statements relevant
to its real roots of {\rm :}
\begin{enumerate}

\item when $D>0$, all three roots of $p(t)$ are real and distinct.
\item when $D=0$, all roots are real;
\begin{itemize}
    \item when $CC'\not=0$, there is one double and one single root;
    \item and when $C=C'=0$, there is one triple root.
\end{itemize}
\item when $D<0$, all three roots are distinct, one of them being a real number  and the other two --- conjugate complex numbers not lying in ${\Bbb R}$.

\end{enumerate}
    
\end{lemma}

\section{Classification of extK germs}

We establish two theorems that provide classifications for generic and exceptional germs, respectively. Consequently, we promptly deduce Theorem 1.1. In addition, we explicitly characterize, in Example 3.3., germs originating from HCMU metrics with cone or cusp singularities on compact Riemann surfaces. 
%Notably, these germs constitute a proper sub-manifold with three components within the broader space encompassed by Theorem 1.1.

%\subsection{Generic germs}

%In order to classify generic HCMU germs at $z=0$, we shall have the following discussion according to Lemma \ref{lem:Card}.

%\subsubsection{When $p_g(t)$ has three distinct real roots}
%In this sub-subsection, we assume that $p_g(t)=-t^3/3+Ct+C'$ has the following three distinct real roots:
%$$K_1>K_2>-(K_1+K_2).$$ This happens iff $D=27\big(4C^3-9(C')^2\big)>0$ by Lemma \ref{lem:Card}. 
%We shall use the following elementary

%\begin{lemma}
%Under the preceding notations, $p_g(t)>0$ iff 
%either $t<-(K_1+K_2)$ or $K_2<t<K_1$.
%\end{lemma}
%\begin{proof} It follows from the graph of $p_g(t)=-t^3/3+\cdots$, which has three distinct real roots:
%$K_1>K_2>-(K_1+K_2)$.
%\end{proof}

%\newpage

\begin{theorem}
\label{lemma:gen_1}
{\rm (Classification of generic germs)} Let $C$ and $C'$ be two real numbers. 

{\rm (1)}\quad The existence of a generic extK germ $g$ satisfying \eqref{eqn:hcmus_o1} near $z=0$, with its Gaussian curvature $K$ attaining $K_0\in \mathbb{R}$ at $z=0$, is contingent upon the condition $-K_0^3/3+CK_0+C' > 0$. Furthermore, such a germ is uniquely determined by $C, C'$ and $K_0$.

{\rm (2)}\quad  The space ${\frak M}_{\rm gx}$ of generic extK germs can be viewed as the domain
\[\left\{(C, C', K_0)\in \mathbb{R}^3 : -K_0^3/3+CK_0+C' > 0 \right\}\] in $\mathbb{R}^3$. Additionally, the mapping
$(C, C', K_0)\mapsto \big(C,\,\ln\,(-K_0^3/3+CK_0+C),\, K_0\big)$
establishes a diffeomorphism from ${\frak M}_{\rm gx}$   onto ${\Bbb R}^3$.

\end{theorem}
\begin{proof}
%This first statement follows from Lemma \ref{lem:Card}. 
(1) Denote $p(t):=-t^3/3+Ct+C$. Recall that such a generic extK germ $g$ and its Gaussian curvature function $K$
satisfy the relation of $g=p(K)|{\rm d}z|^2$ near $z=0$. Restricting it to $z=0$, we find $p\big(K(0)\big)>0$.
Choose $K_0\in {\Bbb R}$ with $p(K_0)>0$, and denote by ${\frak P}(t)$ the primitive of $\frac{1}{p(t)}$ near $K_0$ such
that ${\frak P}(K_0)=0$.  Rewriting the first
equation of \eqref{eqn:hcmus_o1} in the form of 
$\frac{{\rm d}K}{p(K)}={\rm d}\big(z+\bar z\big)$ and integrating it, we find by  ${\frak P}(K_0)=0$ that 
${\frak P}(K)=z+\bar z$ near $z=0$. Since $p(K_0)>0$, we could choose $\delta>0$ such that
$p(t)>0$ and ${\frak P}(t)$ is strictly monotone increasing for all $K_0-\delta\leq t\leq K_0+\delta$.  
Then we obtain
a unique real-valued function $K=K(z,\bar z)$ on the disk 
$$2|z|<\min\big(|{\frak P}(K_0-\delta)|,\, |{\frak P}(K_0+\delta)|\big),$$ where
$K(0)=K_0$, ${\frak P}(K)=z+{\bar z}$ and $p(K)>0$. Hence, the metric $g=4p(K)|{\rm d}z|^2$ on this disc gives
the desired germ. 

(2) is a direct consequence of (1).
\end{proof}

%\begin{remark}
%The same result as the preceding lemma holds even if we replace $\omega={\rm d}z$ by a nontrivial holomorphic one-form near $z=0$ in %\eqref{eqn:hcmus_o1}. 
%\end{remark}

%The validity of this proposition does not depend on the number of real roots of $p_g(t)$. 

%\subsection{Exceptional germs}
\begin{theorem}
\label{lemma:ex}
{\rm (Classification of exceptional germs)}
Let $C,\, C'$ and $\sigma$ be three real numbers and $\sigma\not=0$. 

{\rm (1)}\quad The existence of an exceptional extK germ $g=e^{2\phi}|{\rm d}z|^2$, satisfying \eqref{eqn:hcmus_o2} near $z=0$---where $K=K_g$ and $K_0=K(0)$---is contingent upon the conditions{\rm :}

\begin{equation}
\label{eqn:ex}
\left\{
\begin{split}
-\frac{K_0^3}{3}+CK_0+C'&=0\\
-K_0^2+C&\not=0\\
\frac{1}{-t^3/3+Ct+C'}&-\frac{\sigma}{t-K_0}\  \text{\rm is smooth near}\ t=K_0
%\sigma&=\alpha(C,C',K_0)
\end{split}
\right.\, .
\end{equation}
All such germs forms a real one-parameter family $\big\{g_\lambda:\lambda\in (0,\, +\infty)\big\}$, where 
\[\lambda:=\lim_{z\to 0}\,\frac{({\rm sgn}\,\sigma)\cdot (K_{g_\lambda}-K_0)}{|z|^2}>0\]
is referred to as  the {\rm fifth invariant} 
of germ $(g,\, z)$ at $z=0$. 
$K_{g_\lambda}$ achieves its local minimum (maximum) at $z=0$ if and only if $\sigma>0$ ($\sigma<0$).

{\rm (2)}\quad 
For  all $\lambda\in (0,\,+\infty)$,  these germs of $g_\lambda$ are mutually distinct.

{\rm (3)} \quad 
The space ${\frak M}_{\rm ex}$ of exceptional extK germs can be regarded as the product of $(0,\, +\infty)$ and the real two-dimensional connected sub-manifold consisting of quadruples $(C, C', K_0, \sigma)$ in ${\Bbb R}^3\times {\Bbb R}^\times$ satisfying \eqref{eqn:ex}. Moreover, ${\frak M}_{\rm ex}$ has two components, each of which is diffeomorphic to ${\Bbb R}^3$.

\end{theorem}
\begin{proof}
\noindent (1) We prove at first that \eqref{eqn:ex} is a necessary and sufficient condition for the existence of exceptional germ satisfying \eqref{eqn:hcmus_o2} relevant to a given quadruple $(C,C',K_0,\sigma)\in {\Bbb R}^3\times {\Bbb R}^\times $.\\

%\begin{itemize}

{\sc Necessity}: \quad Let $g=e^{2\phi}|{\rm d}z|^2$ be an exceptional germ satisfying \eqref{eqn:hcmus_o2} relevant to 
quadruple $(C,C',K_0, \sigma)\in {\Bbb R}^3\times {\Bbb R}^\times$, $K$ the Gaussian curvature function of $g$. 
By Lemma \ref{lem:hcmu_sm}, we have $-K_0^3/3+CK_0+C'=0$. We shall prove the necessity in the following three sequential steps.

{\it Step 1}\quad We claim that {\it $C$ and $C'$ do not vanish simultaneously. } Otherwise, $C=C'=K_0=0$, and $g$ and $K$ satisfy
\begin{equation*}
%\label{eqn:hcmus_o2}
\left\{
\begin{split}
%K&=-\frac{4\phi_{z\bar z}}{e^{2\phi}}\\
\partial K&= -\frac{\sigma K^3{\rm d}z}{3z}\\
g&=-\frac{4\sigma^2 K^3}{3}\frac{|{\rm d}z|^2}{|z|^2}=e^{2\phi}|{\rm d}z|^2
\end{split}
\right.\, .
\end{equation*}
Rewriting the first equation of the preceding system as 
\[-3\frac{{\rm d}K}{K^3}=\sigma \left(\frac{{\rm d}z}{z}+ \frac{{\rm d}\bar z}{\bar z}\right)
=\sigma {\rm d}\left(\ln\, |z|^2\right), 
\]
we find that there exists a real constant $D$ such that $\displaystyle{K^2=\frac{3}{2\left(\sigma\ln\, |z|^{2}+D\right)}}$ 
and $\displaystyle{\lim_{z\to 0}\, e^{2\phi(z,\bar z)}=\lim_{z\to 0}\, \frac{-4\sigma^2 K^3}{3|z|^2}=+\infty}$. Contradiction!

{\it Step 2}\quad {\it We shall show  $-K_0^2+C\not=0$.} Otherwise, by Step 1, $K_0\not=0$ must be a double root of the cubic polynomial $-t^3/3+CK+C'$, whose third root equals $-2K_0$. %Moreover, we have  $C=K_0^2$ and $C'=-2K_0^3/3$. 
Since $\displaystyle{\frac{1}{-t^3/3+Ct+C'}=\frac{\alpha}{(K-K_0)^2}+\frac{\beta}{K-K_0}+\frac{\gamma}{K+2K_0}}$ for some $\alpha, \beta,\gamma\in {\Bbb R}$ and $\alpha\not=0$,
we could rewrite the system \eqref{eqn:hcmus_o2} as 
\begin{equation*}
%\label{eqn:hcmus_o2}
\left\{
\begin{split}
%K&=-\frac{4\phi_{z\bar z}}{e^{2\phi}}\\
{\rm d}\big(\sigma\,\ln\,|z|^2\big)&=\left(\frac{\alpha}{(K-K_0)^2}+\frac{\beta}{K-K_0}+\frac{\gamma}{K+2K_0}\right){\rm d}K\\
g&=-\frac{4\sigma^2}{3}\left(K-K_0\right)^2(K+2K_0)\frac{|{\rm d}z|^2}{|z|^2}=e^{2\phi}|{\rm d}z|^2\\
\end{split}
\right.\, .
\end{equation*}
Integrating the first equation of the preceding system, we obtain that as $z\to 0$,  
\begin{eqnarray*}
\frac{\alpha}{K-K_0}&=&-\sigma \ln\,|z|^2+\beta \ln\, |K-K_0|+\gamma\, \ln\, |K+2K_0|+{\rm Const}\\
&=&-\sigma \ln\,|z|^2+o\left(\frac{1}{|K-K_0|}\right),
\end{eqnarray*}
which implies 
$\displaystyle{\lim_{z\to 0} (K-K_0) \cdot (\ln\, |z|^2)=-\alpha/\sigma}.$
Combining it with the second equation of the preceding system, we find $\lim_{z\to 0} e^{2\phi(z,\bar z)}=+\infty$. Contradiction!
We note in passing that Step 1 could also be proved by this argument. 

{\it Step 3}\quad Since $K_0$ is a simple root of $-t^3/3+CK+C'$ by the preceding two steps, there exists a unique nonzero real number $\alpha$ such that
near $t=K_0$, $$\displaystyle{\frac{1}{-t^3/3+CK+C'}-\frac{\alpha}{t-K_0}}$$ equals a smooth function, say $f(t)$.
{\it We shall prove $\sigma=\alpha$.} To this end, we rewrite the system \eqref{eqn:hcmus_o2} as 
\begin{equation*}
%\label{eqn:hcmus_o2}
\left\{
\begin{split}
%K&=-\frac{4\phi_{z\bar z}}{e^{2\phi}}\\
{\rm d}\big(\sigma\,\ln\,|z|^2\big)&=\left(\frac{\alpha}{K-K_0}+f(K)\right){\rm d}K\\
g&=-\frac{4\sigma^2}{3}\left(K-K_0\right)(K^2+K_0K+b)\frac{|{\rm d}z|^2}{|z|^2}=e^{2\phi}|{\rm d}z|^2\\
\end{split}
\right.\, ,
\end{equation*}
where $3C'+bK_0=0$. Integrating the first equation
of the preceding system, we obtain that  $\displaystyle{\lim_{z\to 0}\, |K-K_0|/(|z|^{2})^{\sigma/\alpha}}$ exists as a 
positive number, say $\lambda$. On the other hand,  $\lim_{z\to 0} |K-K_0|/|z|^2$ exists as a positive number
by using the second equation of this system. This implies that $\sigma/\alpha=1$ and $K(z,\overline z)-K_0$ equals
$\pm z\overline z$ times a bounded positive smooth function near $z=0$.  Recalling that 
$\frac{z}{\sigma}=F(z)=4K_{,z}=4e^{-2\phi}K_{\overline z}$, we obtain that {\it $K$ achieves the local 
minimum (maximum) value at $z=0$ if $\sigma>0$} ($\sigma<0$).\\
%We note in passing that the assertion was established that 
%{\it $K$ achieves its local minimum (maximum) if $\sigma>0$} ($\sigma<0$). 

{\sc Sufficiency}:\quad  Choose a quadruple $(C,C',K_0,\sigma)\in {\Bbb R}^4$ satisfying \eqref{eqn:ex}. 
Recall that the desired exceptional germ $g$ should satisfy the system \eqref{eqn:hcmus_o2}, which we could rewrite as  
\begin{equation}
\label{eqn:hcmus_o3}
\left\{
\begin{split}
%K&=-\frac{4\phi_{z\bar z}}{e^{2\phi}}\\
{\rm d}\big(\sigma\,\ln\,|z|^2\big)&=\left(\frac{\sigma}{K-K_0}+f(K)\right){\rm d}K\\
g&=4\sigma^2\left(-\frac{K^3}{3}+CK+C'\right)\frac{|{\rm d}z|^2}{|z|^2}\\
\end{split}
\right.\, ,
\end{equation}
where $f(t)$ is a smooth function near $t=K_0$. Since $K$ achieves its minimum (maximum) $K_0$ at $z=0$ as 
$\sigma>0 (<0)$, there exists $\delta>0$ such that 
$\sigma/(t-K_0)+f(t)$ is strictly positive as $t$ varies in $(K_0,\, K_0+\delta)$ \big($(K_0-\delta,\, K_0)$\big),
where any primitive of $\sigma/(t-K_0)+f(t)$  is strictly monotone increasing and diverges to $-\infty$ ($+\infty$) as $t\to K_0+0$
($t\to K_0-0$). 

Let us assume $\sigma>0$ in the left part of the proof of the sufficiency without loss of generality. 
Choose $\xi\in {\Bbb C}^\times$, $K_\xi\in (K_0,\, K_0+\delta)$, and a primitive $F(t)$ of $\sigma/(t-K_0)+f(t)$ in  $(K_0,\, K_0+\delta)$ which
is strictly monotone increasing and diverges to $-\infty$ as $t\to K_0+0$. Hence the following initial value problem of differential  equation

\begin{equation*}
%\label{eqn:hcmus_o2}
\left\{
\begin{split}
%K&=-\frac{4\phi_{z\bar z}}{e^{2\phi}}\\
{\rm d}\big(\sigma\,\ln\,|z|^2\big)&=\left(\frac{\sigma}{K-K_0}+f(K)\right){\rm d}K\\
K(\xi)&=K_\xi\\
\end{split}
\right.
\end{equation*}
has a unique solution $K=K(z,\bar z)$ in $\{0<|z|\leq |\xi|\}$ such that 
\begin{equation}
\label{eqn:sol_ode}
F(K)=\sigma\,\ln \,|z|^2-\left(\sigma\,\ln\,|\xi|^2-F(K_\xi)\right)
\end{equation}
is a radial function with respect to $|z|$.  
In particular,  $K(z,\bar z)$ is also a radial function in $\{0<|z|\leq |\xi|\}$ and 
$\displaystyle{\lim_{z\to 0}\, (K-K_0)/|z|^2}$ exists as a positive number, say $\lambda$.
Hence we find a desired exceptional germ $g$ living in the closed disk $\{|z|\leq |\xi|\}$ by substituting this solution to the second equation of system \eqref{eqn:hcmus_o3}.  We note in passing that $\lambda$ is termed the {\it fifth invariant} of $g$, whose metric tensor  is also 
radial with respect to $|z|$.\\
%We call the positive limit $\displaystyle{\lim_{z\to 0}\, \frac{({\rm sgn}\,\sigma)\cdot (K-K_0)}{|z|^2}}$ {\it the extra invariant} 
%of germ $g$ associated to the complex coordinate $z$.\\

(2)\quad   To simplify the notations, we shall only consider the case of $\sigma>0$.

\begin{enumerate}
    \item[{\sc Step i}] {\it If two exceptional germs $g_1=e^{2\phi_1}|{\rm d}z|^2$ and $g_2=e^{2\phi_2}|{\rm d}z|^2$ satisfy \eqref{eqn:ex} and have the same fifth invariant, say $\lambda>0$,  then $K_{g_1}\equiv K_{g_2}$ and then $\phi_1\equiv \phi_2$ near $z=0$. In this way, we obtain a family of exceptional germs 
    $\{g_\lambda:\lambda>0\}$.} We shall prove this  by using the notations in the proof of the sufficiency part. Since $K_0$ is a simple root of the cubic polynomial $-t^3/3+Ct+C'$, recalling that $F(t)$ is a primitive
    of $\displaystyle{\frac{1}{-t^3/3+CK+C'}}$ in $(K_0,\, K_0+\delta)$, we find that
    $\sigma\, G(t):=F(t)-\sigma\, \ln\,(t-K_0)$ is smooth and bounded in $[K_0-\delta,\, K_0+\delta]$ as $\delta>0$ is small enough. 
    Choose $\epsilon>0$ and $\xi$ in ${\Bbb D}_\epsilon:=\{|z|<\epsilon\}$, where both of $g_1$ and $g_2$ are defined.
    The problem could be reduced to  showing that $K_{g_1}(\xi)=K_{g_2}(\xi)$. In fact, we rewrite \eqref{eqn:sol_ode} as 
    \[\sigma\,\ln\,(K_{g_1}-K_0)+\sigma\, G(K_{g_1})=\sigma\,\ln \,|z|^2-\Big(\sigma\,\ln\,|\xi|^2-F\big((K_{g_1})(\xi)\big)\Big),\]
    and find that the fifth invariant of $g_1$ equals
    \[\lim_{z\to 0}\,\frac{K_{g_1}-K_0}{|z|^2}=|\xi|^{-2\sigma} \cdot \exp\,\Big(F\big((K_{g_1})(\xi)\big)/\sigma-G(K_0)\Big).\]
    Since $g_1$ has the same fifth invariant as $g_2$ and $F$ is strictly monotone increasing on $(K_0,\, K_0+\delta)$, 
    $K_{g_2}(\xi)=K_{g_2}(\xi)$ by the preceding equation. Hence we obtain that $K_{g_1}\equiv K_{g_2}$ by \eqref{eqn:sol_ode} and then $\phi_1\equiv \phi_2$ in $\{|z|\leq |\xi|\}$ by \eqref{eqn:hcmus_o3}.

    \item[{\sc Step ii}] {\it For any $\lambda$ and $\mu$ in $(0,\,+\infty)$, $g_\lambda$ is isometric to $g_\mu$ near $z=0$ iff $\lambda=\mu$. }
  The proof goes as follows. Suppose that both $g_\lambda$ and $g_\mu$ are defined in ${\Bbb D}_\epsilon$ for some $\epsilon>0$. Then 
  the isometry $\psi:({\Bbb D}_\epsilon,\, g_\lambda)\to ({\Bbb D}_\epsilon,\, g_\mu)$ is conformal and preserves the origin. Hence it must be a rotation with respect to the origin. Then we have $g_\lambda=g_\mu$ since the metric tensor of $g_\mu$ is radial with respect to $|z|$. \\

\end{enumerate}

\noindent (3)\quad  By the proof of (1), the last components $\sigma$ in quadruples
$(C, C', K_0, \sigma)$ in ${\Bbb R}^3\times {\Bbb R}^\times$ satisfying \eqref{eqn:ex}
is a smooth function of the first three components $(C, C', K_0)$. Hence, it suffices to show that the set $\Sigma$ consisting of tuples $(C, C', K_0)\in {\Bbb R}^3$ satisfying 
\[\left\{\begin{split}
0&=-\frac{K_0^3}{3}+CK_0+C'\\
0&\not=-K_0^2+C
\end{split}
\right.\]
is a smooth surface of ${\Bbb R}^3$ diffeomorphic to ${\Bbb R}\times {\Bbb R}^\times$.
Actually, since 
\[\frac{\partial }{\partial K_0}\left(-\frac{K_0^3}{3}+CK_0+C'\right)=-K_0^2+C\]
nowhere vanishes 
on $\Sigma$, by the implicit function theorem, $\Sigma$ is expressed locally as the graph of some smooth function of
$(C,C')$, and it is a real 2-dimensional smooth submanifold of ${\Bbb R}^3$. We claim that {\it the following map
\[\Psi:\Sigma\to {\Bbb R}\times {\Bbb R}^\times,\quad (C,C', K_0)\mapsto (K_0,\, -K_0^2+C) \]
is a diffeomorphism.} Actually, its inverse mapping has the form of
\[\Psi^{-1}(K_0,\,t)=\left(t+K_0^2,\, -\frac{2K_0^3}{3}-tK_0,\,K_0\right).\]

{\it Proof of Theorem 1.1.} \quad Since each extK germ is either generic or exceptional, the space ${\frak M}$ of all extK germs
is the disjoint union ${\frak M}_{\rm gx}$ and ${\frak M}_{\rm ex}$. This theorem follows from Theorems 3.1. and 3.2.   \hfill{QED}

\begin{example} Let us recall at first that 
an {\it HCMU} metric $g$ is a conformal metric with finitely many cone or cusp singularities on a compact Riemann surface such that 
the complex gradient $K_{g,z}\frac{\partial}{\partial z}$ of the Gaussian curvature $K_g$ of $g$ is a holomorphic vector field outside the singularities of $g$. In this discussion, our focus is solely on HCMU metrics with non-constant curvature.
The classification work by Yingyi Wu and the first author \cite[Theorem 2.2]{CW2009} specifically addresses HCMU metrics with cone singularities on compact Riemann surfaces. Wu, the first author, and the third author \cite[Theorems 1.1-7]{CWX2015} extend this to include metrics with  cusp singularities.
Building on their classification theorems, Zhiqiang Wei and Yingyi Wu \cite[Theorem 1.1.]{WW2018} show that these HCMU metrics are pull-backs of HCMU footballs through a multi-valued holomorphic function with monodromy in ${\rm U}(1)$. Consequently, extK germs around smooth points arising from these HCMU metrics precisely coincide with those from HCMU footballs.
In this example, we aim to illustrate that {\it the extK germs arising from HCMU metrics constitute a real 3-dimensional manifold ${\frak O}$ with three components, two of which possess a boundary. Notably, this manifold represents a proper subset within the broader space ${\frak M}$, as delineated in Theorem 1.1.}\\

\noindent{\it Proof.}\quad Let ${\frak O}_{\rm gx}$ represent the subspace of generic germs in $\frak O$,  and ${\frak O}_{\rm ex}$ denote the subspace of
exceptional ones in $\frak O$. 
We describe ${\frak O}={\frak O}_{\rm gx}\sqcup {\frak O}_{\rm ex}$ explicitly as follows:  

\begin{itemize}
\item[(i)] Theorem \ref{lemma:gen_1} (2) says that the space ${\frak M}_{\rm gx}$ of all the generic germs are parametrized by the domain
$\left\{(C,\,C',\,K_0)\in {\Bbb R}^3:-K_0^3/3+CK_0+C'>0  \right\}$
in ${\Bbb R}^3$. 
Wei-Wu \cite[Section 2.2. (8-9)]{WW2018} succinctly outlined the differential systems satisfied by HCMU metrics, both with and without cusp singularities. Examining their summary, it becomes evident that a generic extended Kähler (extK) germ $g$ satisfying \eqref{eqn:hcmus_o1} originates from HCMU metrics if and only if there exist  $K_1 > 0$ and $K_2 \in [-K_1/2, K_1)$ such that
\begin{equation*}
\left\{
\begin{split}
-\frac{t^3}{3}+Ct+C'&=-\frac{1}{3}(t-K_1)(t-K_2)(t+K_1+K_2)\quad {\rm as\ polynomials}\\
K_0&\in (K_2,\, K_1)
\end{split}
\right.\, ,
\end{equation*}
i.e. 
\begin{equation*}
\left\{
\begin{split}
C&=\frac{K_1^2+K_1K_2+K_2^2}{3}\\
C'&=-\frac{K_1K_2(K_1+K_2)}{3}\\
K_0&\in (K_2,\, K_1)
\end{split}
\right.\, .
\end{equation*}
Hence, ${\frak O}_{\rm gx}$ is neither open nor closed in ${\frak M}_{\rm gx}$. Moreover, ${\frak O}_{\rm gx}$ is diffeomorphic to the upper half space
$\{(x,y,t)\in {\Bbb R}^3:t\geq 0\}$ of ${\Bbb R}^3$. 

\item[(ii)]  Theorem \ref{lemma:ex} (2) says that the space ${\frak M}_{\rm ex}$ of all the exception germs $g$ are parametrized all the quintuples
$(C,C',K_0, \sigma,\lambda)\in {\Bbb R}^3\times {\Bbb R}^\times \times (0,\,+\infty)$ satisfying \eqref{eqn:ex}. By \cite[Theorem 2.2.]{CW2009},
such a germ $g$ arises from an HCMU metrics with cone singularities if and only if 
there exist $K_1>0$ and $K_2\in (-K_1/2,\, K_1)$ such that there holds either 
\begin{equation*}
\left\{
\begin{split}
C&=\frac{K_1^2+K_1K_2+K_2^2}{3}\\
C'&=-\frac{K_1K_2(K_1+K_2)}{3}\\
K_0&=K_1\\
\sigma&=-\frac{3}{(K_1-K_2)(K_2+2K_1)}<0
\end{split}
\right.\  {\rm or}\quad  
\left\{
\begin{split}
C&=\frac{K_1^2+K_1K_2+K_2^2}{3}\\
C'&=-\frac{K_1K_2(K_1+K_2)}{3}\\
K_0&=K_2\\
\sigma&=-\frac{3}{(K_2-K_1)(K_1+2K_2)}>0
\end{split}
\right.\, .
\end{equation*}
 By \cite[Theorem 2.2.]{CW2009},
such a germ $g$ arises from an HCMU metric with at least one cusp singularity if and only if 
there exist $K_1>0$ and $K_2=-K_1/2$ such that 
\begin{equation*}
\left\{
\begin{split}
C&=\frac{K_1^2+K_1K_2+K_2^2}{3}\\
C'&=-\frac{K_1K_2(K_1+K_2)}{3}\\
K_0&=K_1\\
\sigma&=-\frac{4}{3K_1^2}<0    
\end{split}
\right.\, .
\end{equation*}
Depending on the sign of $\sigma\not=0$, we divide the space ${\frak O}_{\rm ex}$ of all exception germs from HCMU metrics into
two components, one diffeomorphic to ${\Bbb R}^3$ and a proper open subset of ${\frak M}_{\rm ex}$  corresponding to ${\rm sgn}\,\sigma=1$, 
and the other diffeomorphic to the upper half-space of ${\Bbb R}^3$, and neither open or closed in ${\frak M}_{\rm ex}$ corresponding to ${\rm sgn}\,\sigma=-1$.

\end{itemize}

\end{example}

%Choosing a primitive $F(t)$ of $f(t)/\sigma$ near $t=K_0$ and integrating the first equation of \eqref{eqn:ex}, we obtain 
%a real one-parameter of smooth functions $\{K_\lambda(z,\bar z): 0<\lambda<+\infty\}$ in ${\Bbb D}^\times$ such 
%that for some $\widetilde \lambda>0$, 
%\begin{eqnarray*}
%\sigma\,\ln \,|z|^2= \sigma\, \ln \big(({\rm sgn}\,\sigma)\cdot (K-K_0)\big)-\sigma F(K)-\sigma \ln\,\widetilde \lambda\,.
%\end{eqnarray*}

%We shall specify the relevance between the two positive parameters $\lambda$ and $\widetilde \lambda$ in this equation soon.
%That is,  $K_\lambda=K_0+({\rm sgn}\,\sigma)\, \lambda |z|^2\, e^{F(K_\lambda)}$  in ${\Bbb D}$.
%We note in passing that the value $\lambda$ is determined by the Gaussian curvature $K_\lambda(z_0)$ lying in $(K_0,\, K_0+\delta)$ at any given %point $z_0\in{\Bbb D}^\times$.
%we obtain a one-parameter family of solutions $\{K_\lambda(z,\bar z):\lambda>0\}$ such that 
%${\rm sgn}(\sigma)\cdot (K_\lambda(z,\bar z)-K_0)$ equals $\lambda |z|^2$ times a unique positive smooth function near $z=0$. 
%We obtain the corresponding family $\{g_\lambda:\lambda>0\}$ of exceptional germs by substituting $\{K_\lambda(z,\bar z):\lambda>0\}$ into the %second equation of this system.

\end{proof}

\section{Discussions}
We present the following two open problems pertaining to local extremal K\"ahler metrics in dimension one. 

\begin{prob} Classify {\it complete} extremal Kähler and non-Einstein metrics on open Riemann surfaces, such as ${\Bbb C}$ or ${\Bbb D}={z\in {\Bbb C}: |z|<1}$.
The existence of such a metric on ${\Bbb D}$ is currently unknown, although we lean towards a conjecture of non-existence. The sole known example on ${\Bbb C}$ is the unique HCMU
metric, featuring a cusp singularity at $\infty$ on the Riemann sphere (\cite[Section 8. Appendix]{Chen1999}).
\end{prob}

\begin{prob}
Explore smooth extremal K\"ahler and non-Einstein metrics $g=e^{2\phi}|{\rm d}z|^2$ with {\it finite area} on the punctured disk
${\Bbb D}^\times:=\{0<|z|<1\}$. We introduce the following notations:
$K=K_g=-4\phi_{z\bar z}e^{-2\phi},\quad F(z)=4e^{-2\phi}K_{\bar z},\quad \omega=\frac{{\rm d}z}{F}$.
Applying the reasoning from the initial two paragraphs of the proof of Lemma \ref{lem:hcmu_sm}, we obtain the following assertion:
{\it 
There exist two real constants $C$ and $C'$ such that on ${\Bbb D}^\times$,   
$\omega$ is a meromorphic one-form with at most simple poles, nowhere vanishing, 
and the metric $g$ satisfies}
\begin{equation*}
\left\{
\begin{split}
\partial K&=\left(-\frac{K^3}{3}+CK+C'\right)\omega,\\
g&=4\left(-\frac{K^3}{3}+CK+C'\right)|\omega|^2
\end{split}
\right.\, .
\end{equation*}
Does $\omega$ extend meromorphically to $z=0$? If yes, does $\omega$ have at most a simple pole at $z=0$? If yes again, 
can we characterize their local models near $z=0$ for such extremal K\"ahler metrics on ${\Bbb D}^\times$? 
By employing a sophisticated analysis of the elliptic PDEs relevant to extremal Hermitian metrics \cite{Chen1999, WZ2000},  Chang Shou Lin and Xiaohua Zhu \cite[Section 1]{LZ2002} provided an affirmative answer to the first two questions, assuming that the $L^2$ norm of $K_g$ is finite.
Consequently, they demonstrated that $z=0$ is either a cone or cusp singularity of $g$.
\end{prob}

\noindent{\bf Acknowledgements} 
B.X. extends his heartfelt gratitude to Gao Chen and Yingyi Wu for their invaluable discussions. Specifically, Gao Chen proposed a potential approach involving Functional Analysis, demonstrating the existence of a one-dimensional local K\"ahler metric $g$   that fulfills the conditions $\frac{\partial}{\partial \overline z}K_{g,zz}=0$ and $K_{g,zz}\not=0$.

\bibliographystyle{plain}
\bibliography{RefBase}
\end{document}